\documentclass[12pt]{article}

\usepackage{etex}
\usepackage{amsmath,amssymb,amsfonts,amsthm}
\usepackage{ifpdf}
\usepackage{enumitem}
\usepackage{leftidx}

\usepackage{etoolbox}
\usepackage{fullpage}
\usepackage{longtable}
\usepackage{pdflscape}
\usepackage{multicol}
\usepackage[all]{xy}
\input xy
\xyoption{all}

\usepackage[backend=bibtex,style=alphabetic,doi=false,isbn=false,url=false,minnames=6,maxnames=6]{biblatex}
\renewbibmacro{in:}{%
  \ifentrytype{article}{}{\printtext{\bibstring{in}\intitlepunct}}}
\renewbibmacro*{volume+number+eid}{%
  \printtext{vol.}
  \printfield{volume}%
  \setunit*{\addnbspace}
  \printfield{number}%
  \setunit{\addcomma\space}%
  \printfield{eid}}
\DeclareFieldFormat[article]{number}{\mkbibparens{#1}}

\usepackage{xifthen}  
\usepackage{breqn}
\usepackage{yfonts}
\usepackage{afterpage}

\usepackage{xcolor}
\definecolor{dark-red}{rgb}{0.7,0.25,0.25}
\definecolor{dark-blue}{rgb}{0.15,0.15,0.55}
\definecolor{medium-blue}{rgb}{0,0,0.65}
\definecolor{DarkGreen}{RGB}{0,150,0}

\ifpdf
\usepackage[pdftex,plainpages=false,hypertexnames=false,pdfpagelabels,breaklinks]{hyperref}
\else
\usepackage[dvips,plainpages=false,hypertexnames=false,breaklinks]{hyperref}
\fi
\hypersetup{
   colorlinks, linkcolor={purple},
   citecolor={medium-blue}, urlcolor={medium-blue}
}

\usepackage{tikz}
\usetikzlibrary{calc}
\usetikzlibrary{shapes}
\usetikzlibrary{backgrounds}
\usetikzlibrary{decorations.pathreplacing}
\usepackage{tikz-qtree}

\tikzstyle{shaded}=[fill=red!10!blue!20!gray!30!white]
\tikzstyle{unshaded}=[fill=white]
\tikzstyle{empty box}=[circle, draw, thick, fill=white, opaque, inner sep=2mm]
\tikzstyle{annular}=[scale=.7, inner sep=1mm, baseline]
\tikzstyle{rectangular}=[scale=.75, inner sep=1mm, baseline=-.1cm]

\newcommand{\doi}[1]{\href{https://dx.doi.org/#1}{{\small DOI:#1}}}

\newcommand{\googlebooks}[1]{(preview at \href{https://books.google.com/books?id=#1}{google books})}

\newcommand{\numdam}[1]{}

\theoremstyle{plain}
\newtheorem{prop}{Proposition}[subsection]

\newtheorem{thm}[prop]{Theorem}
\newtheorem{thmalpha}{Theorem}

\newtheorem{coralpha}[thmalpha]{Corollary}

\newtheorem*{cor*}{Corollary}
\newtheorem*{thm*}{Theorem}

\numberwithin{equation}{section}

\theoremstyle{remark}
\newtheorem{example}[prop]{Example}

\newtheorem{remark}[prop]{Remark}           
\newtheorem*{rem*}{Remark}               
\newtheorem*{example*}{Example}                

\theoremstyle{definition}
\newtheorem{defn}[prop]{Definition}         
   
\newtheorem{nota}[prop]{Notation}   
\newtheorem*{defn*}{Definition}             

\theoremstyle{plain}

\newcounter{comment}
\newcommand{\noop}[1]{}

\def\clap#1{\hbox to 0pt{\hss#1\hss}}

\def\semicolon{;}
\def\applytolist#1{
    \expandafter\def\csname multi#1\endcsname##1{
        \def\multiack{##1}\ifx\multiack\semicolon
            \def\next{\relax}
        \else
            \csname #1\endcsname{##1}
            \def\next{\csname multi#1\endcsname}
        \fi
        \next}
    \csname multi#1\endcsname}

\def\calc#1{\expandafter\def\csname c#1\endcsname{{\mathcal #1}}}
\applytolist{calc}QWERTYUIOPLKJHGFDSAZXCVBNM;
\def\bbc#1{\expandafter\def\csname bb#1\endcsname{{\mathbb #1}}}
\applytolist{bbc}QWERTYUIOPLKJHGFDSAZXCVBNM;
\def\bfc#1{\expandafter\def\csname bf#1\endcsname{{\mathbf #1}}}
\applytolist{bfc}QWERTYUIOPLKJHGFDSAZXCVBNM;

\DeclareMathOperator{\Tr}{Tr}

\DeclareMathOperator{\sh}{sh}

\newcommand{\id}{\boldsymbol{1}}
\renewcommand{\imath}{\mathfrak{i}}
\renewcommand{\jmath}{\mathfrak{j}}

\makeatletter
\newcommand{\hashdef}[2]{\@namedef{#1}{#2}}
\newcommand{\hashlookup}[1]{\@nameuse{#1}}
\makeatother

\IfFileExists{../../graphs/lookup.tex}%
	{\newcommand{\pathtographs}{../../graphs/}}%
	{\newcommand{\pathtographs}{diagrams/graphs/}}

\hashdef{bwd1v1p1v1x0p1x1p0x1v0x1x0duals1v1x2v1}{FF05E76A92AAD72C}%
\hashdef{bwd1v1p1v1x0p1x1p0x1v0x1x0duals1v2x1v1}{70D1C9A600E435C3}%
\hashdef{bwd1v1v1p1p1v1x0x0p0x1x0duals1v1v2x1}{63FDB4A04F409318}%
\hashdef{bwd1v1v1v1p1p1v1x0x0p0x1x0p0x0x1v1x0x0p0x1x0p0x0x1duals1v1v1x2x3v1x2x3}{772136FAF425B059}%
\hashdef{bwd1v1v1v1v1p1p1v1x0x0p0x1x0p0x0x1v1x0x0p0x1x0v1x0p0x1duals1v1v1v2x1x3v2x1}{ACE47192690819DF}%

\newcommand{\bigraph}[1]{{\hspace{-3pt}\begin{array}{c}%
  \raisebox{-2.5pt}{\includegraphics[height=6mm]{\pathtographs \hashlookup{#1}}}%
\end{array}\hspace{-3pt}}}

\makeatletter

\def\@testdef #1#2#3{%
  \def\reserved@a{#3}\expandafter \ifx \csname #1@#2\endcsname
 \reserved@a  \else
\typeout{^^Jlabel #2 changed:^^J%
\meaning\reserved@a^^J%
\expandafter\meaning\csname #1@#2\endcsname^^J}%
\@tempswatrue \fi}

\usepackage{libertine}
\usepackage[T1]{fontenc}

\usepackage{titling}
\newcommand{\FS}{\mathfrak{F}_S } 

\newcommand{\unF}{
\begin{tikzpicture}[baseline = -.17cm, xscale=.7, yscale=.6]
  \draw[thick] (-.2,-.325) -- (-.2,.15) -- (.15,.15) -- (.15,.05) -- (-.1,.05) -- (-.1,-.05) -- (.15,-.05) -- (.15,-.15) -- (-.1,-.15) -- (-.1,-.325) -- (-.22,-.325);
\end{tikzpicture}
}

\newcommand{\shF}{
\begin{tikzpicture}[baseline = -.17cm, xscale=.7, yscale=.6]
  \draw[thick, fill=gray] (-.2,-.325) -- (-.2,.15) -- (.15,.15) -- (.15,.05) -- (-.1,.05) -- (-.1,-.05) -- (.15,-.05) -- (.15,-.15) -- (-.1,-.15) -- (-.1,-.325) -- (-.22,-.325);
\end{tikzpicture}
}

\newcommand{\unP}{
\begin{tikzpicture}[baseline = -.17cm, xscale=.7, yscale=.6]
  \draw[thick] (-.2,-.325) -- (-.2,.15) -- (-.05,.15) arc (90:-90:.15cm) -- (-.1,-.15) -- (-.1,-.325) -- (-.22,-.325);
  \draw[thick, unshaded] (-.1,.05) -- (-.05,.05) arc (90:-90:.05cm) -- (-.1,-.05) -- (-.1,.075);
\end{tikzpicture}
}

\newcommand{\shP}{
\begin{tikzpicture}[baseline = -.17cm, xscale=.7, yscale=.6]
  \draw[thick, fill=gray] (-.2,-.325) -- (-.2,.15) -- (-.05,.15) arc (90:-90:.15cm) -- (-.1,-.15) -- (-.1,-.325) -- (-.22,-.325);
  \draw[thick, unshaded] (-.1,.05) -- (-.05,.05) arc (90:-90:.05cm) -- (-.1,-.05) -- (-.1,.075);
\end{tikzpicture}
}

\newcommand{\unS}{
\begin{tikzpicture}[baseline = -.18cm, xscale=.7, yscale=.6]
  \draw[thick] (0,0) arc (0:270:.15cm) arc (90:-180:.05cm) -- (-.3,-.2) arc (-180:90:.15cm) arc (270:0:.05cm) -- (.02,0);
\end{tikzpicture}
}

\newcommand{\shS}{
\begin{tikzpicture}[baseline = -.18cm, xscale=.7, yscale=.6]
  \filldraw[thick, fill=gray] (0,0) arc (0:270:.15cm) arc (90:-180:.05cm) -- (-.3,-.2) arc (-180:90:.15cm) arc (270:0:.05cm) -- (.02,0);
\end{tikzpicture}
}

\newcommand{\unU}{
\begin{tikzpicture}[baseline = -.17cm, xscale=.7, yscale=.6]
  \draw[thick] (-.2,.15) -- (-.2,-.125) arc (-180:0:.2cm) -- (.2,.15) -- (.1,.15) -- (.1,-.125) arc (0:-180:.1cm) -- (-.1,.15) -- (-.22,.15);
\end{tikzpicture}
}

\newcommand{\shU}{
\begin{tikzpicture}[baseline = -.17cm, xscale=.7, yscale=.6]
  \draw[thick, fill=gray] (-.2,.15) -- (-.2,-.125) arc (-180:0:.2cm) -- (.2,.15) -- (.1,.15) -- (.1,-.125) arc (0:-180:.1cm) -- (-.1,.15) -- (-.22,.15);\end{tikzpicture}
}

\newcommand{\unV}{
\begin{tikzpicture}[baseline = -.17cm, xscale=.7, yscale=.6]
  \draw[thick] (-.2,.15) -- (0,-.325) -- (.2,.15) -- (.1,.15) -- (0,-.125) -- (-.1,.15) -- (-.2,.15);
  \draw[thick] (-.15,.15) -- (-.2,.15) -- (-.192,.131);
\end{tikzpicture}
}

\newcommand{\shV}{
\begin{tikzpicture}[baseline = -.17cm, xscale=.7, yscale=.6]
  \draw[thick, fill=gray] (-.2,.15) -- (0,-.325) -- (.2,.15) -- (.1,.15) -- (0,-.125) -- (-.1,.15) -- (-.2,.15);
  \draw[thick] (-.15,.15) -- (-.2,.15) -- (-.192,.131);
\end{tikzpicture}
}

\newcommand{\unZ}{
\begin{tikzpicture}[baseline = -.17cm, xscale=.7, yscale=.6]
  \draw[thick] (.2,-.325) -- (-.2,-.325) -- (-.2,-.225) -- (.05,.05) -- (-.2,.05) -- (-.2,.15) -- (.2,.15) -- (.2,.05) -- (-.05,-.225) -- (.2,-.225) -- (.2,-.348);
\end{tikzpicture}
}

\newcommand{\shZ}{
\begin{tikzpicture}[baseline = -.17cm, xscale=.7, yscale=.6]
  \draw[thick, fill=gray] (.2,-.325) -- (-.2,-.325) -- (-.2,-.225) -- (.05,.05) -- (-.2,.05) -- (-.2,.15) -- (.2,.15) -- (.2,.05) -- (-.05,-.225) -- (.2,-.225) -- (.2,-.348);
\end{tikzpicture}
}

\newcommand{\op}{\text{op}}

\title{Lifting shadings on symmetrically self-dual subfactor planar algebras}
\author{Zhengwei Liu, Scott Morrison, and David Penneys}
\date{}
\begin{document}
\maketitle

\begin{abstract}
In this note, 
we discuss the notion of symmetric self-duality of shaded planar algebras, which allows us to lift shadings on subfactor planar algebras to obtain $\bbZ/2\bbZ$-graded unitary fusion categories.
This finishes the proof that there are unitary fusion categories with fusion graphs 4442 and 3333.
\end{abstract}


Planar algebras have proven to be useful in the construction \cite{MR2679382,MR2979509} and classification \cite{MR3166042,1509.00038} of subfactors and fusion categories.
In recent articles, we used planar algebras to construct subfactor planar algebras with principal graphs 4442, 3333, and 2221 \cite{MR3314808}, 
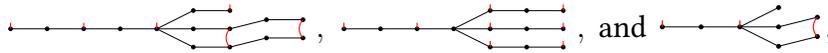
$$
\bigraph{bwd1v1v1v1v1p1p1v1x0x0p0x1x0p0x0x1v1x0x0p0x1x0v1x0p0x1duals1v1v1v2x1x3v2x1}\,,\,
\bigraph{bwd1v1v1v1p1p1v1x0x0p0x1x0p0x0x1v1x0x0p0x1x0p0x0x1duals1v1v1x2x3v1x2x3}\,,\text{ and }
\bigraph{bwd1v1v1p1p1v1x0x0p0x1x0duals1v1v2x1},
$$
and a new subfactor with principal graphs 22221 with interesting dual data \cite{MR3306607} (see Example \ref{ex:22221} below). 
This 22221 subfactor turns out to be an example of a new parameterized family of unshaded subfactor planar algebras related to quantum subgroups \cite{1507.06030}.
The 2221 subfactor was originally constructed by Izumi \cite{MR1832764}, as was the 3333 subfactor \cite{1609.07604}.

In \cite[Appendix]{MR2786219}, Ostrik constructed a $\bbZ/2$-graded unitary fusion category with fusion graph 2221.
Thus upon constructing 4442 and 3333, we naturally wondered whether we could lift the shading on our subfactor planar algebras to get $\bbZ/2$-graded fusion categories with these fusion graphs.

We showed that these subfactor planar algebras are \emph{symmetrically self-dual}, i.e., there is a planar algebra isomorphism $\Phi$ from $\cP_\bullet=(\cP_+,\cP_-)$ to its dual $\overline{\cP}_\bullet=(\cP_-,\cP_+)$ such that $\Phi_\mp\circ \Phi_\pm = \id_\pm$.
We furthermore claimed that we can lift the shading on a symmetrically self-dual subfactor planar algebra to obtain an unshaded factor planar algebra \cite{MR3405915}.
This would mean the associated tensor category of projections is a $\bbZ/2$-graded unitary fusion category $\cC$ whose even graded part is the even part of $\cP_\bullet$ and whose odd graded part is the odd part of $\cP_\bullet$.

In this note, we complete the proof of this claim to complete the construction of these categories.

\begin{thmalpha}\label{thm:LiftShading}
Given a symmetrically self-dual shaded planar algebra $(\shP_\bullet,\Phi)$, there is an unshaded planar algebra $\unP_\bullet$ such that $\shP_\bullet$ is obtained from $\unP_\bullet$ by re-shading (as in Definition \ref{defn:ShadeP} below). 
\end{thmalpha}

\begin{coralpha}
There are $\bbZ/2$-graded unitary fusion categories with fusion graphs 4442 and 3333.
\end{coralpha}

\subsection{Unshaded and shaded planar algebras}\label{sec:Shading}

We refer the reader to \cite{MR2679382,MR2972458} for the definition of a shaded subfactor planar algebra, and to \cite{MR3405915,MR3624399} for the definition of an unshaded factor planar algebra.
As we need to shade unshaded planar algebras and lift the shading on shaded planar algebras, we introduce the following notation to help keep things as simple as possible.

\begin{nota}
We denote an unshaded planar algebra by the unshaded symbol $\unP_\bullet$, and we denote a shaded planar algebra by the shaded symbol $\shP_\bullet$.
We write $\unU$ for an unshaded planar tangle and $\unZ$ for the unshaded partition function.
Similarly, we write $\shS$ for a shaded planar tangle and $\shZ$ for the shaded partition function.
We use the symbol $\cT$ to denote a tangle which may be shaded or unshaded.

Given a shaded tangle $\shS$, let $\shS^\op$ be the shaded planar tangle obtained from $\shS$ by reversing the shading.
Note that forgetting the shading of $\shS$ gives an unshaded tangle $\unS$.
In Definition \ref{defn:ShadeT} below, given an unshaded tangle $\unU$, we define a special shaded tangle $\shU$.
\end{nota}

\begin{remark}
We caution the reader that for an unshaded planar algebra $\unP_\bullet$, the space $\unP_n$ corresponds to boxes with $n$ strands, whereas for a shaded planar algebra $\shP_\bullet$, the space $\shP_n$ corresponds to boxes with $2n$ strands.
\end{remark}

\begin{defn}
Given a shaded planar algebra $\shP_\bullet$, its \emph{dual planar algebra} $\overline{\shP}_\bullet$ is the planar algebra obtained by reversing the shading. 
The planar algebra $\shP_\bullet$ is \emph{self-dual} if there is a shaded planar algebra isomorphism $\Phi: \shP_\bullet\to \overline{\shP}_\bullet$. 
Note that $\Phi$ can be written as
$$
(\Phi_+,\Phi_-): (\shP_{+},\shP_{-}) \longrightarrow (\shP_{-},\shP_{+}).
$$
This means for a shaded planar tangle $\shS$ with $s$ input disks, $\Phi\circ \shS=\shS^\op \circ (\bigotimes^s \Phi)$.

If moreover $\Phi^2 = \id_{\shP_\bullet}$, which is equivalent to  $\Phi_\mp\circ \Phi_\pm = \id_{\shP_{\pm}}$, then $\shP_\bullet$ is called \emph{symmetrically self-dual} \cite{MR3314808}.
\end{defn}

One source of symmetrically self-dual shaded planar algebras is unshaded planar algebras.

\begin{defn}\label{defn:ShadeP}
Let $\unP_\bullet$ be an unshaded planar algebra such that $\unP_{2n+1}=(0)$ for all $n\geq 0$.
Define $\shP_\bullet$ as follows.
For all $n\geq 0$, let $\shP_{n,\pm} = \unP_{2n}$.
Suppose that for a shaded planar tangle $\shS$, the unshaded tangle $\unS$ is obtained from forgetting the shading of $\shS$.
Notice that since $\shS$ is shaded, $\unS$ has an even number of boundary points on all disks.
We define $\shZ(\shS)=\unZ(\unS)$, i.e., the action of $\shS$ is simply obtained by forgetting the shading.

Finally, note that switching the shading on $\shP_\bullet$ is obviously a symmetric self-duality, since the action $\shZ$ is independent of the shading, and merely depends on the unshaded action $\unZ$.
\end{defn}

Theorem \ref{thm:LiftShading} says that all symmetrically self-dual shaded planar algebras come from unshaded planar algebras. 
This theorem is both unsurprising and relatively tedious to prove.
At this point a number of fusion categories have been discovered first in the form of shaded planar algebras (see \cite{MR3314808, MR3306607}).
The motivation of this article is to complete the construction of these  categories.


To prove Theorem \ref{thm:LiftShading}, we introduce the following notation.

\begin{nota}
Suppose $\cT$ has $t$ input disks.
For $1\leq i \leq t$, we let $D_i(\cT)$ denote the $i$-th input disk of $\cT$, and we let $D_0(\cT)$ denote the output disk of $\cT$.
We write $k_i(\cT)$ to denote the number of boundary points of $D_i(\cT)$.
Finally, if $\cT$ is shaded, we write $\pm_i(\cT)$ to denote whether the distinguished interval of the disk $D_i(\cT)$ is in an unshaded or shaded region respectively.
\end{nota}

\begin{defn}\label{defn:ShadeT}
For an unshaded planar tangle $\unU$ whose input and output disks all have an even number of boundary points, we define a shaded tangle $\sh(\unU)=\shU$ by checkerboard shading $\unU$ using the rule that the region meeting the distinguished interval of the output disk is unshaded.

If  $\unU,\unV$ are unshaded planar tangles such that $\unV$ has $k_0(\unV) = k_i(\unU)$ boundary points, then the shading map $\sh$ behaves as follows:
$$
\sh(\unU\circ_i \unV)=
\begin{cases}
\shU\circ_i \shV &\text{if $\pm_i(\shU)=+$}\\
\shU\circ_i \shV^\op&\text{if $\pm_i(\shU)=-$}
\end{cases}
$$
where $\mp_i(\shU)=- \pm_i(\shU)$.  
\end{defn}

We now define an unshaded planar algebra from a symmetrically self-dual planar algebra.
The proof of Theorem \ref{thm:LiftShading} will then consist of showing the action $\unZ$ is well-defined.

\begin{defn}
Given a symmetrically self-dual shaded planar algebra $(\shP_\bullet,\Phi)$, we define a $\bbZ/2$-graded unshaded planar algebra $\unP_\bullet$ as follows. 
Let
$$
\unP_{n} =
\begin{cases}
\shP_{n/2,+} &\text{if $n$ is even}\\
(0) &\text{if $n$ is odd.} 
\end{cases}
$$ 
We now define the action $\unZ$.
First, given an unshaded tangle $\unU$ so that $k_i(\unU)\in 2\bbZ$ for all $0\leq i \leq t$, define
$$
\psi_{i}(\shU) =\begin{cases}
\id_{k_i(\resizebox{.35cm}{.18cm}{\shU})} &\text{if }\pm_i(\shU) = +\\
\Phi_+ &\text{if }\pm_i(\shU) = -.
\end{cases}
$$ 
We define the action of $\unU$ on $\unP_\bullet$ by
$$
\unZ(\unU)=\shZ\left(\shU\right)\circ \left(\bigotimes_{i=1}^t \psi_{i}(\shU)\right): \bigotimes_{i=1}^t \shP_{k_i,+} \to \shP_{k_0,+}.
$$
Note that this defines a map $\bigotimes_{i=1}^t \unP_{2k_i} \to \unP_{2k_0}$.
\end{defn}

We must now show that $\unP_\bullet$ is well-defined, i.e., that the action of a composite of tangles $\unZ(\unU\circ_i \unV)$ is equal to the composite of the actions of the tangles $\unZ(\unU)\circ_i \unZ(\unV)$.
We first give an illustrative example.

\begin{example}
Let $\unF$ be the unshaded 1-click rotation tangle, and let $\shF : \shP_{k,-} \to \shP_{k,+}$ be the shaded 1-click rotation.
Note that $\unZ(\unF)=\shZ(\shF) \circ \Phi_+$. 
Then
\begin{align*}
\unZ(\unF)\circ \unZ(\unF) 
&=\left(\shZ(\shF)\circ\Phi_+\right)\circ \left(\shZ(\shF)\circ\Phi_+\right)\displaybreak[1]\\
&=\left(\shZ(\shF)\circ\Phi_+\right)\circ \left(\Phi_- \circ \shZ(\shF^\op)\right)\displaybreak[1]\\
&=\shZ(\shF) \circ (\Phi_+\circ \Phi_-)\circ \shZ(\shF^\op)\displaybreak[1]\\
&=\shZ(\shF) \circ \shZ(\shF^\op)\displaybreak[1]\\
&=\shZ(\shF\circ \shF^\op)\displaybreak[1]\\
&=\shZ(\sh(\unF\circ\unF))\displaybreak[1]\\
&=\unZ(\unF\circ\unF).
\end{align*}
\end{example}

\begin{proof}[Proof of Theorem \ref{thm:LiftShading}]
Suppose $\unU,\unV$ are unshaded planar tangles with $u,v$ input disks respectively, with $u\geq 1$, such that $\unV$ has $k_0(\unV) = k_i(\unU)$ boundary points.
We show that $\unZ(\unU\circ_i \unV) = \unZ(\unU)\circ_i \unZ(\unV)$.
We treat the two cases for $\pm_i(\shU)$ separately.

\item[]\underline{Case 1:} 
Suppose that $\pm_i(\shU)=+$, so $\psi_i(\shU)=\id_{k_i(\shU)}$. 
Then we have
\begin{align*}
\unZ(\unU)\circ_i \unZ(\unV)
&=
\left(\shZ(\shU)\circ \left(\bigotimes^u_{i=1} \psi_{\ell}(\shU)\right)\right) \circ_i \left(\shZ(\shV) \circ \left(\bigotimes^{v}_{j=1} \psi_{j}(\shV)\right)\right)
\displaybreak[1]\\&=
\shZ(\shU\circ_i \shV)\circ \left(\bigotimes^{u+v-1}_{j=1} \psi_{j}(\shU\circ_i\shV)\right) 
\displaybreak[1]\\&= 
\unZ(\unU\circ_i \unV).
\end{align*}

\item[]\underline{Case 2:}
Suppose that $\pm_i(\shU)=-$, so $\psi_i(\shU)=\Phi_+$.
First, note that for all $1\leq j \leq v$,
\begin{align*}
 \Phi_{\pm_j(\shV)}\circ \psi_{j}(\shV)\
&=
\begin{cases}
\Phi_-\circ \Phi_+ = \id_{k_j,+} &\text{if }\pm_j(\shV)=+\\
\Phi_+\circ \id_{k_j,+} = \Phi_+ &\text{if }\pm_j(\shV)=-\\
\end{cases}
\Bigg\}
=
\psi_{j}(\shV^\op).
\end{align*}
Then we have that
\begin{align*}
 \unZ(\unU) &\circ_i \unZ(\unV)
\displaybreak[1]\\&=
\left(\shZ(\shU)\circ\left( \bigotimes_{\ell=1}^u \psi_{\ell}(\shU)\right)\right)
\circ_i 
\left(\shZ(\shV)\circ\left( \bigotimes_{j=1}^v \psi_{j}(\shV)\right)\right)
\displaybreak[1]\\&=
\left(\shZ(\shU)\circ\left( \bigotimes_{\ell\neq i} \psi_{\ell}(\shU)\right)\right)
\circ_i 
\left(\Phi_+\circ
\shZ(\shV)\circ\left( \bigotimes_{j=1}^v \psi_{j}(\shV)\right)\right)
\displaybreak[1]\\&=
\left(\shZ(\shU)\circ\left( \bigotimes_{\ell\neq i} \psi_{\ell}(\shU)\right)\right)
\circ_i 
\left(\Phi_+\circ \shZ(\shV)\circ
\left( \bigotimes_{j=1}^v \Phi_{\mp_j(\shV)}\right)
\circ
\left( \bigotimes_{j=1}^v \Phi_{\pm_j(\shV)}\circ \psi_{j}(\shV)\right)
\right)
\displaybreak[1]\\&=
\left(\shZ(\shU)\circ\left( \bigotimes_{\ell\neq i} \psi_{\ell}(\shU)\right)\right)
\circ_i 
\left(
\shZ(\shV^\op)
\circ
\left( \bigotimes_{j=1}^v \psi_{j}(\shV^\op)\right)
\right)
\displaybreak[1]\\&=
\shZ(\shU\circ_i \shV^\op)\circ \left(\bigotimes^{u+v-1}_{j=1} \psi_{j}(\shU\circ_i\shV^\op)\right) 
\displaybreak[1]\\&=
\unZ(\unU\circ_i \unV).
\qedhere
\end{align*}
\end{proof}

\begin{remark}
Definition \ref{defn:ShadeP} and Theorem \ref{thm:LiftShading} show that there is a bijective correspondence between symmetric self-dualities on subfactor planar algebras and unshaded $\bbZ/2\bbZ$-graded factor planar algebras.
This gives one way to construct examples of $\bbZ/2\bbZ$-graded extensions of the principal even half $\cC_0$ of a finite depth subfactor planar algebra by its odd half $\cC_1$.

The complete answer to this question was given by \cite{MR3354332}, using the techniques of \cite{MR2677836}.
In particular, \cite{MR3354332} gives a cohomological condition under which any invertible self-dual $\cC_0$-bimodule category $\cC_1$ gives exactly 2 such extensions of the form $\cC_0\oplus \cC_1$.

It is possible there might be extensions of the form $\cC_0\oplus \cC_1$ which are not generated by symmetrically self-dual objects. 
However, we have no such examples at this time.
\end{remark}

\subsection{Examples}\label{sec:Examples}

\begin{example}
In \cite{MR3314808}, we showed that several spoke subfactors are symmetrically self-dual, in particular, the 4442, 3333 ($3^{\bbZ/2\bbZ\times \bbZ/2\bbZ}$), and 2221 subfactors with principal graphs
$$
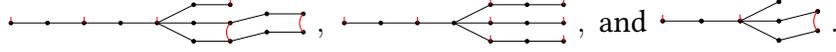
\bigraph{bwd1v1v1v1v1p1p1v1x0x0p0x1x0p0x0x1v1x0x0p0x1x0v1x0p0x1duals1v1v1v2x1x3v2x1}\,,\,
\bigraph{bwd1v1v1v1p1p1v1x0x0p0x1x0p0x0x1v1x0x0p0x1x0p0x0x1duals1v1v1x2x3v1x2x3}\,,\text{ and }
\bigraph{bwd1v1v1p1p1v1x0x0p0x1x0duals1v1v2x1}\,.
$$
\end{example}

\begin{example}
\label{ex:22221}
Many braided subfactor planar algebras coming from BMW algebras \cite{MR936086,MR3592517} are really unshaded.
In \cite{MR3306607}, we introduced a notion of $\sigma$-braided subfactor planar algebras, and we showed our examples with principal graphs
$$
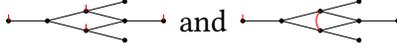
\bigraph{bwd1v1p1v1x0p1x1p0x1v0x1x0duals1v1x2v1}
\text{ and }
\bigraph{bwd1v1p1v1x0p1x1p0x1v0x1x0duals1v2x1v1}
$$
are symmetrically self-dual.
The former one is a special case of BMW planar algebras and $\sigma=\pm1$. The latter one is new and $\sigma=\pm i$.
Their $\mathbb{Z}_2$ fixed point planar algebras have principle graphs $2^{\mathbb{Z}_2 \times \mathbb{Z}_2}1$ and $2^{\mathbb{Z}_4}1$ respectively (22221).
\end{example}

\begin{example}
The $\sigma=\pm i$ subfactor in Example \ref{ex:22221} turns out to be a special case of a new parameterized family of unshaded subfactor planar algebras \cite{1507.06030} related to quantum subgroups.
\end{example}

\begin{example}\label{ex:TY}
For a finite group $G$, the group subfactor $R\subset R\rtimes G$ is self-dual if and only if $G$ is abelian.
For $G$ abelian, the group subfactor $R\subset R\rtimes G$ has many symmetric self-dualities, which we parametrize in \S \ref{sec:bicharacters} below via symmetric bicharacters.
The resulting unitary fusion categories obtained via Theorem \ref{thm:LiftShading} are Tambara-Yamagami categories \cite{MR1659954}.
 
Recall that given an abelian group $A$, a non-degenerate symmetric bicharacter $\chi: A\times A\to S^1$, and a sign $\pm$, the Tambara-Yamagami category $\cT\cY(A,\chi,\pm)$ is a skeletal category with objects $a\in A$ and an object $m$ satisfying the following fusion rules:
\begin{enumerate}[label=(\arabic*)]
\item
$a\otimes b = ab$ for all $a,b\in A$
\item
$a\otimes m=m\otimes a = m$ for all $a\in A$, and
\item 
$m\otimes m = \bigoplus_{a\in A} a$.
\end{enumerate}
It was shown in \cite{MR1659954} that the categories $\cT\cY(A,\chi,\pm)$ and $\cT\cY(A',\chi',\pm')$ are equivalent if and only if there is an isomorphism $A\to A'$ which sends $\chi\to \chi'$, and the signs are equal: $\pm=\pm'$.
(See also \cite[Example 9.4]{MR2677836}.)

For example, Tambara-Yamagami showed in \cite[Theorem 4.1]{MR1659954} that there are only two inequivalent symmetric bicharacters on $\bbZ/2\bbZ \oplus \bbZ/2\bbZ$, which gives 4 inequivalent categories.
Similarly, it is straightforward to show that there are only 2 inequivalent symmetric bicharacters on $\bbZ/4\bbZ$, which gives 4 inequivalent categories.

By \cite{MR3021796},
all Tambara-Yamagami categories are unitary,
since they are weakly group theoretical.
However, to get a factor planar algebra, we need that $m$ is symmetrically self-dual, i.e., $m$ has Frobenius-Schur indicator 1 \cite{MR2381536}.
The Frobenius-Schur indicators for Tambara-Yamagami categories were completely worked out in \cite{MR2774703}, where it was shown that $\nu_2(a)=\delta_{a^2,e}$ and $\nu_2(m)=\pm$, the sign in $\cT\cY(A,\chi,\pm)$.
Hence we must have $\pm=+$ to get a factor planar algebra.
\end{example}

\begin{example}
The $m$-interval Jones-Wassermann subfactors for modular tensor categories are symmetrically self-dual for all $m\geq 1$ \cite{1612.08573}. 
This result was inspired by the open problem to construct a conformal net whose representation category is a prescribed unitary modular tensor category.
We refer the reader to  \cite{1612.08573} for further references related to Jones-Wassermann subfactors of conformal nets.

In particular, if we start with a unitary modular tensor category whose fusion ring is a finite abelian group $A$, then its $m$-interval Jones-Wassermann subfactor, for any $m\geq 2$, gives a symmetrically self-dual group subfactor (with group $A$), as well as a Tambara-Yamagami category as in Example \ref{ex:TY}. 
Moreover, the symmetric bicharacter $\chi$ of $A$ is determined by the modular $S$ matrix of the modular tensor category. 
\end{example}

\subsection{Group subfactors: symmetrical self-dualities and symmetric bicharacters}
\label{sec:bicharacters}

Given an abelian group $A$, we obtain a group subfactor planar algebra $\shP^A_\bullet$.
The minimal projections of $\shP^A_{2,\pm}$ are indexed by the groups elements in $A$.
These projections correspond to invertible objects in the category of projections \cite{MR2559686,MR3405915}, where the fusion corresponds to the coproduct, denoted $*$.
We denote the minimal projections of $\shP^A_{2,+}$ by $P_g$ for $g\in A$ so that 
$d P_g*P_h=P_{gh}$, where $d=\sqrt{|A|}$.

The group subfactor planar algebra $\shP^A_\bullet$ is an exchange relation planar algebra \cite{MR1950890}. 
By \cite[Theorem 2.26]{MR3551573},  $\Phi: \shP^A_\bullet\to \overline{\shP^A}_\bullet$ is a shaded planar algebra $*$-isomorphism of exchange relation planar algebras if and only if $\Phi$ preserves the structure of 2-boxes $\shP^A_{2,\pm}$, which consists of the adjoint operator $(\cdot)^*$, the trace, the contragredient (rotation by $\pi$), the multiplication (stacking), and the coproduct $*$.   

\begin{thm}\label{Thm:self-dual}
Let $A$ be a finite abelian group, and let $\FS:\shP^A_{2,+}\to \shP^A_{2,-}$ be the string Fourier transform (one-click rotation).  
If $\chi(g,h)$ is a non-degenerate bicharacter on $A$, then 
$$
\Phi_-(\FS(P_g)):=\sum_{h\in A} \frac{1}{d}\chi(g,h)P_h
\qquad\qquad\qquad
\Phi_+:=\FS^{-1}\Phi_-\FS
$$
extends uniquely to a shaded planar algebra $*$-isomorphism $\Phi: \shP^A_\bullet\to \overline{\shP^A}_\bullet$. 

Conversely, if $\Phi: \shP^A_\bullet\to \overline{\shP^A}_\bullet$ is a shaded planar algebra $*$-isomorphism, then the coefficients $\chi(g,h)$ defined above give a bicharacter on $A$.
\end{thm}

\begin{proof}
We define $\Phi_\pm$ as in the statement and extend the map linearly to labelled tangles in the universal planar algebra generated by 2-boxes. 
By  \cite[Theorem 2.26]{MR3551573}, it is enough to prove that $\Phi_-$ preserves the adjoint operator, the trace, the contragredient, the multiplication, and the coproduct on 2-boxes.

\item[]\underline{Step 1:} 
the extension $\Phi_-$ preserves the adjoint operator:
$$
(\Phi_-(\FS(P_g)))^*
=\sum_{h\in A} \frac{1}{d} \overline{\chi(g,h)}P_h
=\sum_{h\in A} \frac{1}{d} \chi(g^{-1},h)P_h
=\FS(P_{g^{-1}})
=\FS^{-1}(P_{g})
=(\FS(P_{g}))^*.
$$

\item[]\underline{Step 2:} 
the extension $\Phi_-$ preserves the trace:
$$
\Tr(\Phi_-(\FS(P_g)))
=\sum_{h\in A} \frac{1}{d}\chi(g,h)
=d \delta_{g=e}
=\Tr(\FS(P_g)),
$$
where $\delta_{g=e}$ is the Kronecker function, and $e$ is the identity of the group $A$. 

\item[]\underline{Step 3:} 
the extension $\Phi_-$ preserves the contragredient:
\begin{align*}
\overline{\Phi_-(\FS(P_g))}
&=\sum_{h\in A} \frac{1}{d}\chi(g,h)\overline{P_h}\\
&=\sum_{h\in A} \frac{1}{d}\chi(g,h)P_{h^{-1}}\displaybreak[1]\\
&=\sum_{h\in A} \frac{1}{d}\chi(g,h^{-1})P_h\\
&=\sum_{h\in A} \frac{1}{d}\overline{\chi(g,h)}P_{h}\displaybreak[1]\\
&=\sum_{h\in A} \frac{1}{d}\chi(g^{-1},h)P_{h}\\
&=\FS(P_{g^{-1}})\displaybreak[1]\\
&=\FS(\overline{P_{g}}).
\end{align*}

\item[]\underline{Step 4:} 
the extension $\Phi_-$ preserves the multiplication:
\begin{align*}
\Phi_-(\FS(P_{g_1}))\Phi_-(\FS(P_{g_2}))
&=\sum_{h\in A} \frac{1}{d^2}\chi(g_1,h)\chi(g_2,h)P_h\\
&=\sum_{h\in A} \frac{1}{d^2}\chi(g_1g_2,h)P_h\\
&=\frac{1}{d}\Phi_-(\FS(P_{g_1g_2}))\\
&=\Phi_-(\FS(P_{g_1}* P_{g_2}))\\
&=\Phi_-(\FS(P_{g_1}) \FS(P_{g_2})).
\end{align*}

\item[]\underline{Step 5:} 
the extension $\Phi_-$ preserves the coproduct:
\begin{align*}
\Phi_-(\FS(P_{g_1}))*\Phi_-(\FS(P_{g_2})
&=\sum_{h_1,h_2\in A} \frac{1}{d^2}\chi(g_1,h_1) \chi(g_2,h_2)P_{h_1}*P_{h_2}\\
&=\sum_{h_1,h\in A} \frac{1}{d^3}\chi(g_1,h_1) \chi(g_2,h_1^{-1}h)P_{h}\\
&=\sum_{h_1,h\in A} \frac{1}{d^3}\chi(g_1g_2^{-1},h_1) \chi(g_2,h)P_{h}\\
&=\sum_{h_1\in A} \frac{1}{d^2}\chi(g_1g_2^{-1},h_1) \FS (P_{g_2})\\
&=\delta_{g_1,g_2} \FS (P_{g_2})\\
&= \FS (P_{g_1}P_{g_2})\\
&= \FS (P_{g_1})*\FS(P_{g_2}).
\end{align*}
Therefore $\Phi$ extends uniquely to a shaded planar algebra $*$-isomorphism $\Phi: \shP^A_\bullet\to \overline{\shP^A}_\bullet$.
Conversely, if $\Phi: \shP^A_\bullet\to \overline{\shP^A}_\bullet$ is a shaded planar algebra $*$-isomorphism, then by definition,
$$
\chi(g,h)
=d \Tr(\Phi_-(\FS(P_g))P_h).
$$ 
In particular $\chi(e,h)=\Tr(P_h)=1$. By the computation in Step 3,
$$
\Phi_-(\FS(P_{g_1}))\Phi_-(\FS(P_{g_2}))=\Phi_-(\FS(P_{g_1})\FS(P_{g_2}))
$$ 
implies that 
$\chi(g_1g_2,h)=\chi(g_1,h)\chi(g_2,h).$
Thus $\chi(g,h)$ is a character with respect to $g$.
On the other hand, 
$$
\chi(g,h)=d \Tr(\Phi_-(\FS(P_g))P_h)
%
=
d \Tr(\FS(P_g)\Phi_-^{-1}(P_h))
$$
and in particular $\chi(g,e)=d \Tr(\FS(P_g) P_e)=1$.
Moreover, 
\begin{align*}
\chi(g,h_1h_2)
&=d \Tr(\FS(P_g)\Phi_-^{-1}(P_{h_1h_2}))\\
&=d \Tr(P_g\FS^{-1}(\Phi_-^{-1}(P_{h_1h_2})))\\
&=d^2 \Tr(P_g\FS^{-1}(\Phi_-^{-1}(P_{h_1}*P_{h_2})))\\
&=d^2 \Tr(P_g\FS^{-1}((\Phi_-^{-1}(P_{h_1})*\Phi^{-1}(P_{h_2})))\\
&=d^2 \Tr(P_g\FS^{-1}(\Phi_-^{-1}(P_{h_1})) \FS^{-1}(\Phi^{-1}(P_{h_2})))\\
&=d^2 \Tr(P_g\FS^{-1}(\Phi_-^{-1}(P_{h_1}))) \Tr(P_g \FS^{-1}(\Phi^{-1}(P_{h_2})))\\
&=\chi(g,h_1)\chi(g,h_2).
\end{align*}
Therefore $\chi(g,h)$ is a bicharacter.
\end{proof}

\begin{thm}\label{Thm:symmetrically self-dual}
The bicharacter $\chi(g,h)$ in Theorem \ref{Thm:self-dual} is symmetric if and only if $\Phi^2=1$, namely $\shP^A_{\bullet}$ is symmetrically self-dual.
\end{thm}

\begin{proof}
Note that $\Phi_+\Phi_-=\FS^{-1}\Phi_-\FS\Phi_-$, so $\Phi^2=1$ if and only if $(\Phi_-\FS)^2=\FS^2$.

If the bicharacter $\chi(g,h)$ is symmetric, then 
$$
(\Phi_-\FS)^2(P_g)
=\sum_{h,k \in A}\frac{1}{d^2} \chi(h,k) \chi(g,h) P_k
=\sum_{h,k \in A}\frac{1}{d^2} \chi(g+k,h) P_k
=P_{g^{-1}}
=\FS^2(P_{g}).
$$
Thus $(\Phi_-\FS)^2=\FS^2$ and $\Phi^2=1$.

On the other hand, if $\Phi^2=1$, then 
\begin{align*}
\chi(g,h)
&=d \Tr(\Phi_-(\FS(P_g))P_h) \\
&=d \Tr(\FS^{-1} (\Phi_-(\FS(P_g))) \FS(P_h)) \\
&=d \Tr(\Phi_-^{-1}(P_g) \FS(P_h)) \\
&=d \Tr(P_g \Phi_-(\FS(P_h))) \\
&=d \Tr(\Phi_-(\FS(P_h))P_g) \\
&=\chi(h,g),
\end{align*}
and the bicharacter is symmetric.
\end{proof}

For a finite abelian group $A$ and a non-degenerate symmetric bicharacter $\chi: A\times A\to S^1$, 
the $\mathbb{Z}/2$-graded unitary fusion category corresponding to the unshaded planar algebra $\shP^A_{\bullet}$ is the Tambara-Yamagami category. 
We refer the readers to \cite{1612.08573} for a more general case, where $A$ is replaced by a unitary modular tensor category and $\chi$ is replaced by the modular $S$ matrix.

\subsection*{Acknowledgements}

We would like to thank Noah Snyder for helpful conversations.
Zhengwei Liu was supported by a grant from Templeton Religion Trust and an AMS Simons travel grant.
Scott Morrison was supported by Discovery Projects `Subfactors and symmetries' DP140100732 and `Low dimensional categories' DP160103479, and a Future Fellowship `Quantum symmetries' FT170100019 from the Australian Research Council.
David Penneys was supported by an AMS Simons travel grant and NSF DMS grants 1500387/1655912 and 1654159.

\renewcommand*{\bibfont}{\small}
\setlength{\bibitemsep}{0pt}
\raggedright
\printbibliography

\end{document}